\documentclass[12,a4paper]{amsart}

\usepackage{amsmath,latexsym}
\usepackage{amssymb}
\usepackage[english, activeacute]{babel}
\usepackage[latin1]{inputenc}
\usepackage{float}
\usepackage{hyperref}
\usepackage{amsmath}
\usepackage{amssymb}
\usepackage{graphicx}

\newcommand{\N}{\mathbb{N}}
\newcommand{\R}{\mathbb{R}}

\newcommand{\rank}{\mathrm{rank}\,}
\newcommand{{\km}}{\rm k}
\newcommand{\dsum}{\displaystyle\sum}
\def\B{\mathcal{B}}

\def\Mo{\mathrm{M}}
\def\L{\mathrm{L}}

\newtheorem{theo}{Theorem}
\newtheorem{defi}[theo]{Definition}
\newtheorem{ex}[theo]{Example}
\newtheorem{lem}[theo]{Lemma}
\newtheorem{cor}[theo]{Corollary}

\title{A semidefinite programming approach for solving\\Multiobjective Linear Programming}

\author{V\'ictor Blanco}
\address{Dep. of Quantitative  Methods for Economics \& Business, Universidad de Granada.}
\email{vblanco@ugr.es}
\author{Justo Puerto \and Safae Ben-Ali}
\address{IMUS, Universidad de Sevilla.}
\email{puerto@us.es,anasafae@gmail.com}


\begin{document}
\maketitle

\begin{abstract}
Several algorithms are available in the literature for finding the entire set of Pareto-optimal solutions in MultiObjective Linear Programming (MOLP). However, it has not been proposed so far an interior point algorithm that finds all Pareto-optimal solutions of MOLP. We present an explicit construction, based on a transformation of any MOLP into a finite sequence of SemiDefinite Programs (SDP), the solutions of which give the entire set of Pareto-optimal extreme points solutions of MOLP. These SDP problems are solved by interior point methods; thus our approach provides a pseudo-polynomial interior point methodology to find the set of Pareto-optimal solutions of MOLP.
\end{abstract}

\section{Introduction}
Although already more than 60 years old, Linear Programming is still nowadays one of the most important areas of research and application in Mathematical Programming/Operations Research. There are many different algorithms for its resolution but essentially all of them can be classified in four main branches: primal simplex, dual simplex, primal-dual simplex and interior point methods.

Although not as the same level as standard linear programming, it is also commonly accepted that Multiobjective Linear programming (MOLP) is another very important area of activity within the optimization field. Multiobjective optimization is motivated by the need to consider multiple, conflicting, objectives in real world decision making problems. Multiple objective linear programming has been a subject of research since the 1960s for its
relevance in practice, as a mathematical topic in its own right. Its development has come in parallel to the scalar counterpart and its theory and algorithms have been developed mainly in the last two decades with   some remarkable exceptions. Solving a MOLP means to obtain the entire set of Pareto-optimal solutions. By analogy with the scalar case, one can find in the specialized literature several algorithms to find the entire Pareto-optimal set. In the  multiobjective case there are primal simplex-like type algorithms (Steuer \cite{Steuer1985}, Yu and Zeleny \cite{YZ1975} and the references therein), primal-dual simplex-like algorithms (Ehrgott, Puerto and Rodriguez-Chia \cite{EPR2007}) and dual simplex-like algorithms (Benson \cite{Benson1998}, Ehrgott, Lohne and Shao \cite{ELS2011}). Moreover, there are some partial approaches that use interior point methods to  approximate or to find some Pareto-optimal points (see e.g. Fliege \cite{Fliege2004,Fliege2006}) but no interior point algorithm that finds all Pareto-optimal solutions has been proposed so far. Quoting \cite{ELS2011}: \textit{``Interior point algorithms either require interaction with a decision maker to find a preferred efficient solution or find at most an efficient face. No interior point algorithm that finds all efficient solutions has been proposed''.}

Due to the natural parallelism between these two areas, namely scalar and multiobjective linear programming, different authors wondered whether there would exist a  fully based interior point approach to generate the complete Pareto-optimal set of a MOLP. This paper  gives an affirmative answer to this question.

The goal of this paper is to develop an interior point method that generates the entire Pareto-optimal set of MOLP. The interest of this research is two-fold: 1) It strengthens the parallelism between these two areas of Mathematical Programming, namely Linear Programming and Multiobjective Linear Programming; 2) It is theoretically appealing because shows how to adapt some tools available nowadays in the field of polynomial optimization (\cite{lasserrebook,lasserre22}) to be applied in a completely different field as is Multiobjective Linear Programming.

From the results in this paper, we show that it is theoretically possible to find the entire set of Pareto-optimal solutions using  interior point algorithms. Indeed, we will explicitly describe an algorithm to perform that task. Our approach  consists of constructing  a finite sequence (its cardinality $m$ is the number of constraints in the MOLP) Semidefinite programs based on the original MOLP, the solution of which can be used to generate the entire Pareto-optimal set using the moment matrix algorithm by Lasserre \cite{lasserrebook}. Nevertheless, this program size is not polynomial in the input size of our original MOLP. There is no surprise in this fact because it is well-known that finding the entire Pareto-optimal set is NP-hard since it may be equivalent to enumerate the vertices of its feasible region Khachygan et al \cite{BBEG08}. Thus if our construction were polynomial in the input size, interior point algorithms would prove  polynomiality of the problem. Needless to say that even though our construction is explicit, we do not claim that this approach is computationally competitive with other approaches available nowadays, as for instance the variant of the outer algorithm of Benson developed recently in \cite{ELS2011}. Nevertheless, it is of theoretical interest because proposes the first interior point based algorithm for obtaining the entire Pareto-optimal set, it suggests a completely different scheme to address this problem and proposes another  pseudo-polynomial algorithm to deal with that computation.  It is worth mentioning a different framework where a similar approach is applicable: multiobjective polynomial integer programming.  We point out that the use of algebraic tools for solving multiobjective is not new. The interested reader is  referred to \cite{BP2009,BP2011} for further details.

The rest of the paper is organized as follows. Section \ref{s:molp} describes the general multiobjective optimization problem, together with the considered concept of solution and it recalls the main results that are needed for the rest of the paper. Section \ref{s:sdpmo} presents the theoretical results that lead to obtain the entire set of Pareto-optimal solutions of MOLP using interior point algorithms. Here, we explicitly describe a system of polynomial equations that encodes the entire set of Pareto-optimal extreme points which is the basis of our next results. Section \ref{s:cuatro} reduces the problem of finding the entire set of Pareto-optimal extreme points to solving $m$ explicit SDP problems. They are  feasibility problems since their objective functions are constant. In addition, we also show how to extract all its, finitely many, feasible solutions by applying the moment matrix algorithm \cite{lasserrebook}. This construction is illustrated with an example taken from the literature \cite{ELS2011}. In the final section of the paper we draw some  conclusions.

\section{Multiobjective linear programming\label{s:molp}}

In this section we recall the main theoretical results for the development in this paper. We begin by describing the general framework to cast the problem to be handled. A \emph{multiobjective optimization problem (MOP)} consists of:
\begin{align}
& v-\min \;\;(f_1(x), \ldots, f_k(x))\tag{${\rm MOP}$}\label{mop}\\
&s.t. \;\; x \in S\nonumber
\end{align}
where $f_i: \R^n \rightarrow \R$ for $i=1,\dots,k$ are the \emph{objective functions} and $S \subseteq \R^n$ is the \emph{feasible region}. The symbol $v-\min$ means that we want to minimize all the objective functions simultaneously. If there is no conflict between the objective functions, then, a solution can be found where every objective function attains its optimum. In such a case, no special methods are needed. Otherwise, first we need to state what we understand by a solution of the above problem. It is commonly agreed that the solution concept is related with the notion of Pareto-optimal points. The definition of Pareto-optimal solution is due to Edgeworth \cite{edgeworth}. However, that name was first used by Koopmans in \cite{koopmans} after the developments by Pareto in \cite{pareto64,pareto71}.

\begin{defi}
\label{def:1}
A decision vector $x^* \in S$ is a \emph{Pareto-optimal solution} for \eqref{mop} if it does not exist another decision vector $x\in S$ such that $f_i(x) \leq f_i(x^*)$ for all $i=1, \ldots, k$ and $f_j(x) < f_j(x^*)$ for at least one index $j$.

The set of Pareto-optimal solutions is called the \emph{Pareto-optimal set}. (PO set, for short.)

If $x^*$ is a Pareto-optimal solution $f(x^*)$ is said an \emph{efficient solution} of \eqref{mop}.
\end{defi}
It is commonly agreed that solving a (MOP) consists of finding the entire set of Pareto-optimal solutions. There are some other solution concepts for multiobjective optimization problems as local, weak, proper or strong Pareto optimality (see \cite{miettinen}). In this paper we will restrict ourselves to the standard definition of Pareto-optimal solutions, although similar approaches, to the one adopted in this paper,  would be also valid for the rest of the solution concepts in multiobjective optimization.

One of the main methods for describing the Pareto-optimal set of a multiobjective optimization problem is by \emph{scalarization}, that is, transforming the multiobjective problem into a single or a family of single-objective problems with a real-valued objective function, depending on some parameters. This enables the use of the theory and methods of scalar optimization to be applicable to get the solutions of MOP. The importance of these methods rests on the fact that Pareto-optimal solutions of MOP can be  characterized, in most cases, as solutions of certain single objective optimization problems. There are several of these scalarization methods for solving MOP (see \cite{miettinen}). Among them, we consider the weighting method. The idea of this method is to associate each objective function with a weighting coefficient so as to minimize the weighted sum of the objective functions. The weighting method can be used so that the decision maker specifies a weighting vector representing his preference information. However, this method can also be used to generate iteratively several solutions of \eqref{mop} by modifying adequately  the weighting coefficients.
The success of this approach is based on the following result by Gass and Saaty \cite{gass-saaty} or Zadeh \cite{zadeh63}.

\begin{lem}
\label{lem:2}
Let $f_i$ be convex for all $i=1, \ldots, k$ and $S$ be a convex set. Then, if $x^* \in S$ is a Pareto-optimal solution of \eqref{mop}, there exists a weighting vector $\lambda \in \R^k_+\setminus \{0\}$, $\sum_{i=1}^k\lambda_i=1$ such that $x^*$ is a solution of the following scalar problem:
\begin{align}
 \min &\; \dsum_{i=1}^k \lambda_if_i(x) \tag{${\rm SP}$}\label{sp}\\
s.t.&\; x \in S\nonumber
\end{align}
\end{lem}

Note that from the above result, in particular if $f_i$ is linear for all $i=1, \ldots, k$ and $S$ is a convex polyhedron, all the Pareto optimal solutions of \eqref{mop} can be found by the weighting method.

In this paper we are interested in solving a special class of multiobjective optimization problems where all the objective functions are linear and the feasible region is described by a set of linear inequalities, that is:
\begin{align}
 v-\min &\; Cx:=(c^1x, \ldots, c^kx) \label{molp}\tag{${\rm MOLP}$}\\
s.t. &\;  Ax \geq b  \nonumber \\
& x \geq 0\nonumber
\end{align}
with $c^i$ the i-the row of $C\in \mathbb{R}^{k\times n}$, $i=1,\ldots,k$, $A \in \R^{m \times n}$ and $b\in \R^m$.  We assume w.l.o.g. that the system $Ax\ge b$, $x\ge 0$, has not redundant inequalities and that $0\not \in cone (c^1,\ldots,c^k)$ since otherwise all the feasible region is Pareto-optimal and the problem is trivial. Under our assumption, the Pareto-optimal set is included in the boundary of the feasible region and it is well-known that it is edge connected but not necessarily  convex, see e.g. \cite{Steuer1985}.   With this settings we say that a Pareto-optimal solution for \eqref{molp} is an \textit{extreme point} if it is a vertex of the polyhedron defining the feasible region of \eqref{molp}.

In addition and in order to simplify our presentation, we will assume w.l.o.g. that the feasible region is a polytope and that we are given redundant upper bounds on the values of the $x$ variables, namely we are given $ub_j^P$ such that $x_j\le ub_j^P$, $i=1,\ldots,n$. Note that by the fact that the feasible region is a polytope these bounds can always be obtained for sufficiently large $ub^P$ values and  they are redundant.

Lemma  \ref{lem:2} ensures that  to solve \eqref{molp} it suffices to apply the above weighting method. In doing that, this problem is transformed to a family of parametric linear programming problems. For this reason, we recall here some results about linear programming that will be useful in the next sections.

Consider the following pair of dual linear programming problems:

 \begin{minipage}[c]{6cm}{\begin{align}
 \min &\; c^tx \nonumber \\
 s.t. & \; Ax \geq b\label{lp}\tag{${\rm LP}$}\\
& x \geq 0\nonumber
\end{align}} \end{minipage}
\begin{minipage}[c]{6cm}{\begin{align}
 \max &\; b^tu\nonumber\\
s.t.& \; A^t u \leq c\label{dlp}\tag{${\rm DLP}$}\\
& u \geq 0\nonumber
\end{align}}
\end{minipage}\\
with $A \in \R^{m \times n}$, $b\in \R^m$ and $c \in \R^n$.

By our assumption on the feasible region of the primal problem (\ref{lp}), we observe that  the dual problem can be also assumed to be bounded and that we can assume that we are also given redundant upper bounds on the feasible values of (DLP). Namely, we know $ ub_i^D$ such that $u_i\le ub_i^D$ for all $i=1,\ldots,m$.

The following classical result, whose proof can be found in \cite{vanderbei}, gives the relationship between the optimal solutions of the above problems:
\begin{lem}[Strong Duality Theorem]
\label{lem:3}
Let $x^*$ be a feasible solution of \eqref{lp} and let $u^*$ be a feasible solution of \eqref{dlp} such that $cx^* = b^tu^*$, then $x^*$ is an optimal solution for \eqref{lp} and $u^*$ is an optimal solution for \eqref{dlp}.
\end{lem}

An important consequence of the \textit{Strong Duality Theorem} is the following result, usually called  \emph{Complementary Slackness Theorem}. It  is useful in order to detect whether a feasible solution of \eqref{lp} and a feasible solution of \eqref{dlp} are optimal of their respective problems. Its proof, that can be found in \cite{schrijver}, is easily deduced from Lemma \ref{lem:3}.
\begin{lem}[Complementary Slackness Property]
\label{lem:4}
Let $x^*$ be a feasible solution of \eqref{lp} and let $u^*$ be a feasible solution of \eqref{dlp}. Then, the following statements are equivalent:
\begin{enumerate}
\item $x^*$ is an optimal solution of \eqref{lp} and  $u^*$ is an optimal solution of \eqref{dlp}.
\item $x^*$ and $u^*$ satisfy $u^{*t} (b - Ax)=0$ and $ (u^tA - c^t)x^*=0$.
\end{enumerate}
\end{lem}

In the next section we will show how the entire set of Pareto-optimal solutions of a multiobjective linear problem  \eqref{molp} can be obtained by solving $m$-SDP problems being $m$ the number of non-redundant inequalities of the original problem. This shows that interior point methodologies can be also used to determine the Pareto-optimal set of \eqref{molp}.

\section{A polynomal system of inequalities encoding the Pareto-optimal set of MOLP\label{s:sdpmo}}

Note that by Lemma \ref{lem:2}, solving \eqref{molp} is equivalent to solve the following parametric family of single objective problems:
\begin{align}
\tag{${\rm LP}_\lambda$}  \min\; & \dsum_{\ell=1}^k \lambda_{\ell} c^{\ell} x \label{lpl}\\
 \mbox{s.t. } &  Ax \geq b \nonumber \\
& x \geq 0\nonumber
\end{align}
for all $\lambda \in \Lambda:=\{ \lambda \in \R^k_+:\; \sum_{\ell}^k \lambda_{\ell}=1\}$.

For each $\lambda \in \Lambda$, the dual of \eqref{lpl} is:
\begin{align}
\tag{${\rm DLP}_\lambda$} \max\;& \dsum_{j=1}^m u_jb_j\label{dlpl}  \\
\mbox{s.t. } & u^tA\leq \dsum_{i=1}^k \lambda_i c^i \nonumber\\
 &u \geq 0\nonumber
\end{align}

Hence, by Lemma \ref{lem:4}, a solution of \eqref{molp} must be a solution of the following system of polynomial equations/inequalities:
\begin{align}
& u^t (b-Ax)= 0\nonumber\\
& (\dsum_{i=1}^k \lambda_i c^i - u^tA) x = 0\nonumber\\
& Ax \geq b\label{system1}\tag{${\rm Sys1}$}\\
& u^tA\leq \dsum_{i=1}^k \lambda_i c^i\nonumber\\
& u, x\ge 0,\; \; \lambda \in \Lambda \nonumber
\end{align}
Solutions of this system are triplets $(x,u,\lambda)$ such that $x$ is a Pareto-optimal solution of \eqref{molp}, and optimal solution of $(PL_{\lambda})$; and $u$ is optimal solution of $(DLP_{\lambda})$. We will call such a triplet a \textit{`valid triplet'}.
However, we observe that the above system may have an infinite number of solutions because there may be a continuum number of solutions in $x$ and $u$ since the Pareto-optimal set is a connected union of faces of the polyhedron $Ax\ge b$, $x\ge0$. Nevertheless by the edge-connectedness of the Pareto-optinal set, it suffices to know the  Pareto-optimal points that are extreme points in the feasible region to reconstruct the entire set (see for instance \cite{ArmandMalivert91,ecker-kouada}). Therefore, our goal is to reduce the solutions of the above system to a finite number of the original ones (extreme points) that are enough to reconstruct the entire set of Pareto-optimal solutions. Hence, we shall transform  System \eqref{system1} so as to characterize only extreme points of the original feasible region.

It is worth mentioning that a similar approach could be applied to general multiojective convex optimization. Indeed, in that case we have to replace the use of System (\ref{system1}) by  using Kuhn-Tucker optimality conditions on the parametric family of weighted problems. The only inconvenient is that in that case the system may have a continuum of solutions and in general, one cannot build the entire Pareto-optimal set with a finite set of representatives, as we do in the linear case. This implies  that  finiteness results cannot be obtained. Another different framework where a similar approach is applicable is in multiobjective polynomial combinatorial optimization. The interested reader is referred to \cite{BP2009,BP2011} for further details.

As it is usual, let $B$ denote a basis of $A$, i.e. a full rank submatrix of $A$. In addition and when there is no possible confusion, we will also use $B$ as the set of the indices of its columns. Analogously,   $N$ denotes the set of columns of $A$ not in $B$ and $c_B^{\ell}$ are the coefficients of the $\ell$-th objective function that corresponds to variables in the basis $B$. Finally, with $x_B$ and $x_N$, we refer to the variables corresponding to the columns in $B$ and $N$, respectively.

Let us denote by \eqref{usysbi} for  $B$ a full rank submatrix of $A$ and $i=1,\ldots,m$, the system:
\begin{align}
\tag{\rm U-Sys-B-i} \quad \dsum_{\ell=1}^k \lambda_{\ell} c^{\ell}-u_{-i}^tA_{-i}\geq A_{i\cdot} & \label{usysbi}\\
\mbox{\hspace*{1.5cm}}  \quad \dsum_{\ell=1}^k \lambda_{\ell} c^{\ell}_BB^{-1}A_{.j}-\dsum_{\ell=1}^k \lambda_{\ell} c_j^{\ell}\le 0,& \forall j\in N \nonumber \\
\sum_{\ell =1}^k \lambda_{\ell} =1 \nonumber
\end{align}
where $u_{-i}$, $A_{-i}$ stand, respectively, for the vector of variables $u$ without $u_i$ and the matrix $A$ where we have removed the $i$-th row. In the same way, $A_{i\cdot}$ is the $i$-th row, $A_{\cdot j}$ is the $j$-th column and $A_{ij}$ is the $(i,j)$ element of $A$, respectively.

Finally, let $M$ be the least common multiple of all the determinants of full rank submatrices of $A$ and $M_i$ be the least common multiple of all the determinants of full rank submatrices of \eqref{usysbi}, for all $i=1,\ldots,m$ and $B$.

\begin{theo} \label{t:MOLP}
If $x$ is a Pareto-optimal solution, extreme point of the feasible region of \eqref{molp} then $Mx$ is the projection onto the first $n$-components of a solution of at least one of the systems \eqref{systemi}, for some $i=1,\ldots,m$:

\begin{minipage}{10cm}
\begin{align}
& h_0^0(\lambda):= \sum_{\ell =1}^k \lambda_{\ell}-M_i=0\\
& h_1^0(x,u):=u^t (Mb-Ax)= 0\nonumber\\
& h_2^i(x,u):=(\dsum_{\ell=1}^k \lambda_{\ell} c^{\ell} - u_{-i}^tA_{-i}-M_iA_{i\cdot}) x = 0\nonumber\\
& g_s^0(x):=A_{s\cdot}x-Mb_s \geq 0,\; s=1,\ldots,m, \label{c:axb}\\
& g_j^i(u,\lambda):=\dsum_{\ell=1}^k \lambda_{\ell} c^{\ell}_j-u_{-i}^t(A_{-i})_{\cdot j}-M_iA_{ij}\ge 0,\; j=1,\ldots,n, \label{c:lamu}\\
& p_j(x):= \prod_{\ell=0}^{ub_j^PM} (x_j-\ell)=0, \quad j=1,\ldots,n,\nonumber \\
& q_s^i(u):= \prod_{\ell=0}^{ub_s^DM_i} (u_s-\ell)=0, \quad s=1,\ldots,m, \; s\neq i,\nonumber \\
& t_r^i(\lambda):=\prod_{\ell=0}^{M_i} (\lambda_r-\ell)=0, \quad r=1,\ldots,k.\nonumber
\end{align}
\end{minipage}
\begin{minipage}{1cm}
$$
\left\}\begin{array}{c}
\\
\\
\\
\\
\\
\\
\\
\\
\\
\\
\\
\\
\\
\\
\\
\\
\\
\end{array}\right.
$$
\end{minipage}
\begin{minipage}{1.75cm}
 \begin{equation}
\label{systemi}\tag{${\rm Sys-i}$}
\end{equation}
\end{minipage}

Conversely, any of the finitely many solutions of the systems \eqref{systemi} for all $i=1,\ldots,m$ induces a Pareto-optimal solution of \eqref{molp} and all the Pareto-optimal extreme points are included among them.
\end{theo}
\begin{proof}
Let $x$ be a Pareto-optimal extreme point solution of \eqref{molp}. Being an extreme point of the feasible region $Ax\ge b,\; x\ge 0$ means that there exists a basis $B$ which defines $x$ such that $x=\left(\begin{array}{c} x_B \\ 0 \end{array}\right)$ and $x_B=B^{-1}b\ge 0$. Next, clearly $Mx=\left(\begin{array}{c} Mx_B \\ 0 \end{array}\right)$, has integer coordinates in the range $[0,ub^PM]$. Besides, by Lemma \ref{lem:2}, $x$ must be optimal for the problem $(PL_{\lambda(x)})$ for some $\lambda(x)$ (Note that this weighting coefficient may depend on $x$). The optimality condition of $x$ translates into the following necessary and sufficient condition for any valid $\lambda(x)$:
\begin{eqnarray} \label{eq:co-lam}
 \sum_{\ell =1}^k \lambda_{\ell} c_B^{\ell} B^{-1} A_{\cdot j}-\sum_{\ell=1}^k \lambda_{\ell} c_j^{\ell}\le 0; & \quad  \forall\; j\in N,\\
 \sum_{\ell =1}^k \lambda_{\ell} =1, & \label{eq:sum-lam}
\end{eqnarray}
where as usual $N$ denotes the set of columns of $A$ not in $B$ and $c_B^{\ell}$ are the coefficients of the $\ell$-th objective function that correspond to variables in the basis $B$.

Moreover, by Lemma \ref{lem:4}, each pair $(x,\lambda(x))$, where $\lambda(x)$ is defined by (\ref{eq:co-lam})-(\ref{eq:sum-lam}),  has associated an optimal extreme solution of $(DLP_{\lambda(x)})$ such that $(x,u(x),\lambda(x))$ is a valid triplet, i.e. it satisfies System (\ref{system1}). By the fact that $u(x)$ is an extreme solution of  $(DLP_{\lambda(x)})$ it must satisfy:
\begin{equation}
\label{eq:ulam} u^tA- \dsum_{\ell=1}^k \lambda_{\ell} c^{\ell}\le 0
\end{equation}
Assume w.l.o.g. that for some valid triplets $(x,u(x),\lambda(x))$ the $i$-th component of $u(x)$ is positive, i.e. $u(x)_i> 0$, then there is always one of these valid triplets  for $x$ so that $(u(x),\lambda(x))$ is  an extreme solution of the system of linear inequalities defined by (\ref{eq:co-lam}), (\ref{eq:sum-lam}) and (\ref{eq:ulam}). This system written in matrix form is:
\begin{equation*}
L_i\left[ \begin{array}{c} u_{-i} \\ \lambda \end{array} \right]:=\left[ \begin{array}{cc} \Theta & C_BB^{-1}N-C_N \\ A_{-i}^t & C \\ \underline{0} & \underline{1} \end{array} \right] \left[ \begin{array}{c} u_{-i} \\ \lambda \end{array} \right] \begin{array}{c} \le \\ \le \\ = \end{array} \left[ \begin{array}{c}  0 \\  A_{i\cdot}^t \\ 1 \end{array} \right] .
\end{equation*}
Next, any extreme solution of this system must be associated with a basis of the matrix $L_i$. Therefore if $M_i$ is the least common multiple of the determinants of all full rank submatrices of $L_i$, then $(M_i (u_{-i},1), M_i\lambda)$ has integer coordinates in the range $[0,ub^D M_i]$ for $u$ and $[0,M_i]$ for $\lambda$.

Finally, it is clear by our construction that $(Mx,M_iu(x),M_i\lambda(x))$ is a solution of  \ref{systemi}.

Conversely, by Lemma \ref{lem:4}, any solution $(x,u(x),\lambda(x))$ of \ref{systemi} for some $i=1,\ldots,m$ defines a Pareto-optimal triplet $(x/M,u(x)/M_i,\lambda(x)/M_i)$ of \eqref{molp}. In addition, we have proved above that all Pareto-optimal extreme point solutions of \eqref{molp} are among the solutions of the systems \eqref{systemi} for $i=1,\ldots,m$ which concludes the proof.
\end{proof}

The above transformation makes use of big upper bounds to ensure that extreme points of all systems of rational inequalities that come from feasibility and optimality of valid triplets are integer. In most cases, those bounds can be strengthened taking advantage of the particular structure of the problems given rise to sharper bounds. In particular if the primal, the dual or both are integer polytopes one can take $M=1$, $M_i=1$, for all $i=1,\ldots,m$ or both $M=M_i=1$ for all $i=1,\ldots,m$; respectively. Eventually, we will need only to transform the range of the lambda variables to make them integer.

To finish this section, we briefly remark the meaning of Theorem \ref{t:sdpmo}. We have proved that the entire set of Pareto-optimal extreme points is encoded in the set of solutions of $m$ system of polynomial inequalities, projecting the first $n$ components of their solutions. (Observe that this set may include  some extra Pareto-optimal solutions not being extreme points.) The finiteness of the solution sets of these systems is crucial to develop an exact method to get the entire Pareto-optimal set, as we will show in the next section.

\section{Semidefinite programming versus Multiobjetive Linear programming\label{s:cuatro}}
In this section we describe how to obtain the Pareto-optimal set of any MOLP by applying tools  borrowed from the theory of moments and SDP.
We use standard notation in the field (see e.g. \cite{lasserrebook}).

We denote by $\mathbb{R}[x,u,\lambda]$ the ring of real polynomials in the variables $x=(x_1,\ldots,x_n),\; u=(u_1,\ldots,u_m),\; \lambda=(\lambda_1,\ldots,\lambda_k)$; and by $\mathbb{R}[x,u,\lambda]_d \subset \R[x,u,\lambda]$ the space of polynomials of degree at most $d \in \N$ (here $\N$ denotes the set of nonnegative integers). We also denote by $\B = \{x^\alpha u^\beta \lambda^\gamma: (\alpha,\beta,\gamma)\in\mathbb{N}^{n\times m\times k}\}$ a
canonical basis of monomials for $\R[x,u,\lambda]$, where $x^\alpha u^\beta \lambda^\gamma = x_1^{\alpha_1} \cdots x_n^{\alpha_n}u_1^{\beta_1} \cdots u_m^{\beta_m}\lambda_1^{\gamma_1} \cdots \lambda_k^{\gamma_k}$, for any $(\alpha,\beta,\gamma) \in \N^{n\times m\times k}$.

For any sequence indexed in
the canonical monomial basis $\B$, $\mathbf{y}=(y_{\alpha\beta\gamma})\subset\mathbb{R}$, let $\L_\mathbf{y}:\mathbb{R}[x,u,\lambda]\to\mathbb{R}$ be the linear functional defined, for any $f=\sum_{\alpha\beta\gamma\in\mathbb{N}^{n\times m\times k}}f_{\alpha\beta\gamma}\,x^\alpha u^\beta \lambda^\gamma \in \R[x,u,\lambda]$, as $\L_\mathbf{y}(f) :=
\sum_{\alpha\beta\gamma\in\mathbb{N}^{n\times m\times k}}f_{\alpha\beta\gamma}\,y_{\alpha\beta\gamma}$.

The \textit{moment} matrix $\Mo_d(\mathbf{y})$ of order $d$ associated with $\mathbf{y}$, has its rows and
columns indexed by $(x^\alpha u^\beta \lambda^\gamma)$ and $\Mo_d(\mathbf{y})(\alpha\beta\gamma,\alpha' \beta' \gamma')\,:=\,\L_\mathbf{y}(x^{(\alpha+\alpha')}u^{(\beta+\beta')}\lambda^{(\gamma+\gamma')})\,=
\,y_{(\alpha+\alpha')(\beta+\beta')(\gamma+\gamma')}$, for $\vert\alpha\beta\gamma\vert,\,\vert\alpha'\beta'\gamma'\vert\,\leq d.$

For $g:=\sum_{\delta\in \N^{n\times m\times k}} g_\delta (xu\lambda)^\delta\in \mathbb{R}[x,u,\lambda]$, the \textit{localizing%
} matrix $\Mo_d(g, \mathbf{y})$ of order $d$ associated with $\mathbf{y}$
and $g$, has its rows and columns indexed by $((xu\lambda)^\delta)$ and $\displaystyle \Mo_d(g,\mathbf{y})(\alpha\beta\gamma,\alpha' \beta' \gamma'):=\L_\mathbf{y}(x^{(\alpha+\alpha')}u^{(\beta+\beta')}\lambda^{(\gamma+\gamma')}%
g(xu\lambda))=
\sum_{\delta\in \N^{n\times m\times k}}g_\delta y_{\delta+(\alpha+\alpha')(\beta+\beta')(\gamma+\gamma')}$, for $\vert\alpha\beta\gamma\vert,\vert\alpha' \beta' \gamma'\vert\,\leq d$.

Depending on their parity, let $2 \zeta_j$ or $2\zeta_j-1$ be the degree of $p_j$, $j=1,\ldots, n$;  $2 \eta_s^i$ or $2\eta_s^i-1$ be the degree of $q_s^i$, $s=1,\ldots, m$, $i=1,\ldots,m,\; s\neq i$ and $2 \nu_r^i$ or $2\nu_r^1-1$ be the degree of $t_r^i$, $r=1,\ldots, k$,  $i=1,\ldots,m$. Recall that these functions were defined associated with some of the constraints that appear in (\ref{systemi}). Finally, for any symmetric matrix $P$ by $P \succeq 0$ we refer to $P$ being positive semidefinite.

With this notation, we are in position to present our next result:

\begin{theo}  \label{t:sdpmo}
The entire set of  Pareto-optimal extreme point solutions of \eqref{molp} can be obtained by solving $m$  semidefinite programs (\ref{p:sdpmoj}) $i=1,\ldots,m$, for some $N^*\in \mathbb{N}$.
\begin{align}
\min\; \; & y_0:=1 \label{p:sdpmoj} \tag{${SDP^i-N^*}$}\\
\mbox{s.t. }
 & \Mo_{N^*}(x,u,\lambda)\succeq 0,\nonumber  \\
 & \Mo_{N^*-1}(h_0^0(\lambda))=0, \nonumber \\
 & \Mo_{N^*-1}(h_1^0(x,u))=0,\nonumber \\
 & \Mo_{N^*-1}(h_2^i(x,u,\lambda))=0,\nonumber \\
 & \Mo_{N^*-1}(g_s^0(x))\succeq 0,\; s=1,\ldots,m,\nonumber \\
 & \Mo_{N^*-1}(g_j^i(u,\lambda))\succeq 0,\; j=1,\ldots,n,\nonumber \\
 & \Mo_{N^*-\zeta_j } (p_j(x))= 0, \; \; j=1,\ldots,n,\\
 & \Mo_{N^*-\eta_s^i} (q_s^i(u))= 0, \; \; s=1,\ldots,n,\; s\neq i\\
 & \Mo_{N^*-\nu_r^i} (t_r^i(u))= 0, \; \; r=1,\ldots,k.
\end{align}
\end{theo}
\begin{proof}
Theorem \ref{t:MOLP} proves that the entire set of Pareto-optimal extreme point solutions of \eqref{molp} can be obtained as  the projection of the solutions of systems (\ref{systemi}), $i=1,\ldots,m$. Each of these systems  is defined by a closed, compact semi-algebraic set with finitely many feasible solutions that by compacity, clearly satisfy Putinar's condition (see \cite{lasserrebook}).  
Hence, we can apply \cite[Theorem 6.1]{lasserrebook} to conclude that there exists $N^*< +\infty$ so that all the solutions of System (\ref{systemi}) can be obtained by solving the feasibility SDP problem  (\ref{p:sdpmoj}) for all $i=1,\ldots,m$.
\end{proof}

We note in passing that the finiteness of $N^*$ is explicit and it can be bounded above by some known constants which depend on the input data. The interested reader is referred to \cite{NS2007} and \cite{Schw2004}.

As a direct consequence of the equivalence between MOLP and solving a finite series of $m$-SDP (Theorem \ref{t:sdpmo}); and the fact that SDP can be solved by interior points algorithms we conclude the following corollary.

\begin{cor}
All Pareto-optimal extreme solutions of \eqref{molp} can be obtained using  interior point algorithms.
\end{cor}

It is well-known that solving SDP problems is polynomially doable. Therefore, one may conclude from the above reformulation that obtaining the entire Pareto-optimal set of \eqref{molp} is polynomial. However, we cannot conclude this result from Theorem \ref{t:sdpmo}. The reason being that the dimension of the $m$-SDP problems that need to be solved are not polynomial in the size of the original problem because $N^*$ is finite but its size may be exponential in $n,m,k$. Therefore, the above approach only gives a  pseudo-polynomial approach to obtain the Pareto-optimal set of MOLP programs.

Finally, we want to describe, based on the above results, an explicit   methodology to enumerate the real solutions of the above closed semi-algebraic sets (\ref{systemi}), $i=1,\ldots,m$. (Note that this is equivalent to enumerate all the optimal solutions of a polynomial optimization problem where the objective is a constant; see Theorem \ref{t:sdpmo}.)

In order to do that, first we transform w.l.o.g. (\ref{systemi}) into an algebraic set, in  the standard manner, by simply adding non-positive slack variables $w\in \mathbb{R}^m_-$, $z\in \mathbb{R}^n_-$ to the constraints \eqref{c:axb} and \eqref{c:lamu}, respectively. Thus, from now on we assume that all the constraints in the SDP relaxations (\ref{p:sdpmoj}) are in equation form.

For any $i =1, \ldots, m$, consider $J^i =  \langle h_0^0, h_1^0,h_2^i,g_1^0,\ldots,g_m^0,g_1^i,\ldots,g_n^i,p_1,\ldots,p_n,q_1^i, \ldots,q_m^i,t_1^i,\dots,t_k^i \rangle$ the zero-dimensional ideal in $\R[x,u,\lambda]$, generated by the polynomials defining  the System (\ref{systemi}). Note that the ideal $J^i$ to be zero dimensional  is equivalent to suppose that there is a finite number of solutions of the set of polynomial equations defining the system. This set of solutions is denoted by $V(J^i)$ and it is usually called the variety of $J^i$. The goal is then to compute the points in $V(J^i)$.

In order to do this we resort to the moment approach (see Lasserre et al. \cite{LLR2008b}, Lasserre \cite{lasserrebook}). It is clear that among all the finite probability measures defined on the support of our feasible set $V(J^i)$, those assigning positive probability to all the solutions of $V(J^i)$ give maximal range to the moment matrix $\Mo_{N^*}(x,u,\lambda)$, and for sufficiently large $N^*$ $rank (\Mo_{N^*}(x,u,\lambda))=|V(J^i)|$. Moreover, since $V(J^i)$ is finite, by Theorem \ref{t:sdpmo}, there is an index $N^*$ such that all the SDP relaxations \ref{p:sdpmoj}, $i=1,\ldots,m$, are exact and their solutions are moment sequences of probability measures with support on $V(J^i)$.  Hence, the idea is to translate that condition, i.e. that we look for measures which give positive probability to all points in $V(J^i)$, on the moment variables of our SDP relaxations (\ref{p:sdpmoj}), $i=1,\dots,m$. In this way we will get the points in $V(J^i)$ as solutions of the relaxations. This underlying approach was theoretically justified by Lasserre et al. \cite{LLR2008b} and Lasserre \cite[Theorem 6.6]{lasserrebook} and can be effectively implemented by using the moment matrix algorithm as described in \cite[Algorithm 6.1]{lasserrebook}. The reader is referred to \cite{LLR2008b,Laurent2007} for further details.

Let us consider the quotient space $\R[x,u,\lambda]/J^i$ whose elements are the cosets $[f]  = \{f + q: q \in J\}$ for any $f \in \R[x,u,\lambda]$. Since $J^i$ is zero-dimensional, $\R[x,u,\lambda]/J^i$ is a finite dimensional $\R$-vector space with the usual addition and scalar product. Furthermore, $\R[x,u,\lambda]/J^i$ is an algebra with multiplication $[f][g] = [fg]$. If $\B_{J^i}$ is a basis of $R[x,u,\lambda]/J^i$, then the \textit{multiplication matrix} $\Mo_h$  associated with the multiplication operator $m_h: \R[x,u,\lambda]/J^i \rightarrow \R[x,u,\lambda]/J^i$, $m_h([f]) = [fh]$ for $h \in \R[x,u,\lambda]$, plays and important role for obtaining the points in $V(J^i)$.

Let $d=\max\{ \max_{j=1...n} \zeta_j, \max_{\scriptsize{ \begin{array}{c} s=1...m,\\ i=1...n \end{array}}}   \eta_s^i\}$ the maximum half degree of the polynomials defining System (\ref{systemi}), $i=1,\ldots,m$. Finally, let us denote by $R($\ref{p:sdpmoj}$)$ the feasible region of Problem (\ref{p:sdpmoj}), $i=1,\ldots,m$.

The above discussion allows us to apply to the variety $V(J^i)$, theorems 6.2 and 6.5 in \cite{lasserrebook} and Proposition 4.6 in \cite{LLR2008b}  which leads us to state our final result.

\begin{theo} \label{t:eigen}
For $N^*$ large enough, there exists $d \leq t \leq N^*$ such that:
$$ \rank \Mo_{t}(x,u,\lambda)=\rank \Mo_{t-d}(x,u,\lambda) = |V(J^i)|, \quad \forall (x,u,\lambda)\in R(\mbox{\ref{p:sdpmoj}}),\; \forall i=1,\ldots,n.$$
Moreover, one can obtain the coordinates of all $(x,u,\lambda)\in V(J^i),\; \forall i=1,\ldots,n,$ as the eigenvalues of multiplication matrices.
\end{theo}
\bigskip

The following example, that appears in \cite{ELS2011}, illustrates the methodology proposed in this paper.

\begin{ex}
\label{example}
Consider problem \eqref{molp} with the data:
$$
C=\begin{pmatrix}1 & 0\\0 & 1 \end{pmatrix}, \quad A=\begin{pmatrix} 2&1\\1&1\\1&2\\-1&0\\0&-1\end{pmatrix}, \quad b=\begin{pmatrix}4\\3\\4\\-5\\-5\end{pmatrix}
$$
Observe that the last two constraints refer to the upper bound constraints $x_1\leq 5$ and $x_2\leq 5$, so they are not considered as rows of the matrix $A$ but as the sets of upper bounds in the polynomial constraints $p_j(x)$ in Theorem \ref{t:MOLP}.

We use \texttt{Gloptipoly 3}~\cite{gloptipoly} to formulate the semidefinite problems of Theorem \ref{t:sdpmo} and \texttt{SEDUMI 1.3}~ \cite{sedumi} as the SDP solver. Since in this problem $m=3$, according to Theorem \ref{t:sdpmo}, three semidefinite problems must be solved for each relaxation order $N^*$. The original primal region has integer extreme points. Thus, $M=1$ and we use their original bounds for the range of values in problems \eqref{p:sdpmoj}. Therefore, we must only adapt the range of $\lambda$. For $i=1, 2, 3$ we get that for $M_1=M_2=M_3=6$ and $N^* =4$, the rank condition of Theorem \ref{t:eigen} is satisfied, i.e. $\rank \Mo_4(x,u,\lambda)=\rank \Mo_{1}(x,\mu, \lambda) =2$, and we extract the following solutions of \eqref{p:sdpmoj} (2 for $i=1$, $2$ for $i=2$ and $2$ for $i=3$):
\begin{center}
\begin{tabular}{|c|ccc|ccc|ccc|}\hline
& \multicolumn{3}{c}{$i=1$} &\multicolumn{3}{|c|}{$i=2$} &\multicolumn{3}{|c|}{$i=3$} \\\hline
&$x$ & $u$ & $\lambda$ &$x$ & $u$ & $\lambda$ &$x$ & $u$ & $\lambda$ \\\hline
Sol. \#1&$(1,2)$ & $(2,0,0)$ & $(4,2)$ & $(1,2)$ & $(0,3,0)$ & $(3,3)$& $(2,1)$ & $(0,0,3)$ & $(3,3)$\\
Sol. \#2&$(0,4)$ & $(2,0,0)$ & $(4,2)$ & $(2,1)$ & $(0,3,0)$ & $(3,3)$ & $(4,0)$ & $(0,0,2)$ & $(2,4)$\\
\hline
\end{tabular}
\end{center}
Note that the number of moments involved in the SDP problems that had to be solved was 1716 (923 after preprocessing/moments substitution). The moment matrix $\Mo_{N^*}(x,u,\lambda)$ has size $210\times 210$, there are $6$ localizing matrices with size $84\times 84$ and $4208$ linear inequalities.

Thus, projecting the set of extracted solutions onto the $x$-coordinates, we get the set of Pareto-optimal extreme solutions of the problem, $X_E=\{(4,0), (1,2), (2,1), (0,4)\}$. These Pareto-optimal solutions and the complete Pareto-optimal set are shown in Fig. \ref{fig:exsol} (black dots and black segments, respectively).
\begin{figure}[h]
\begin{center}
 \includegraphics[scale=0.3]{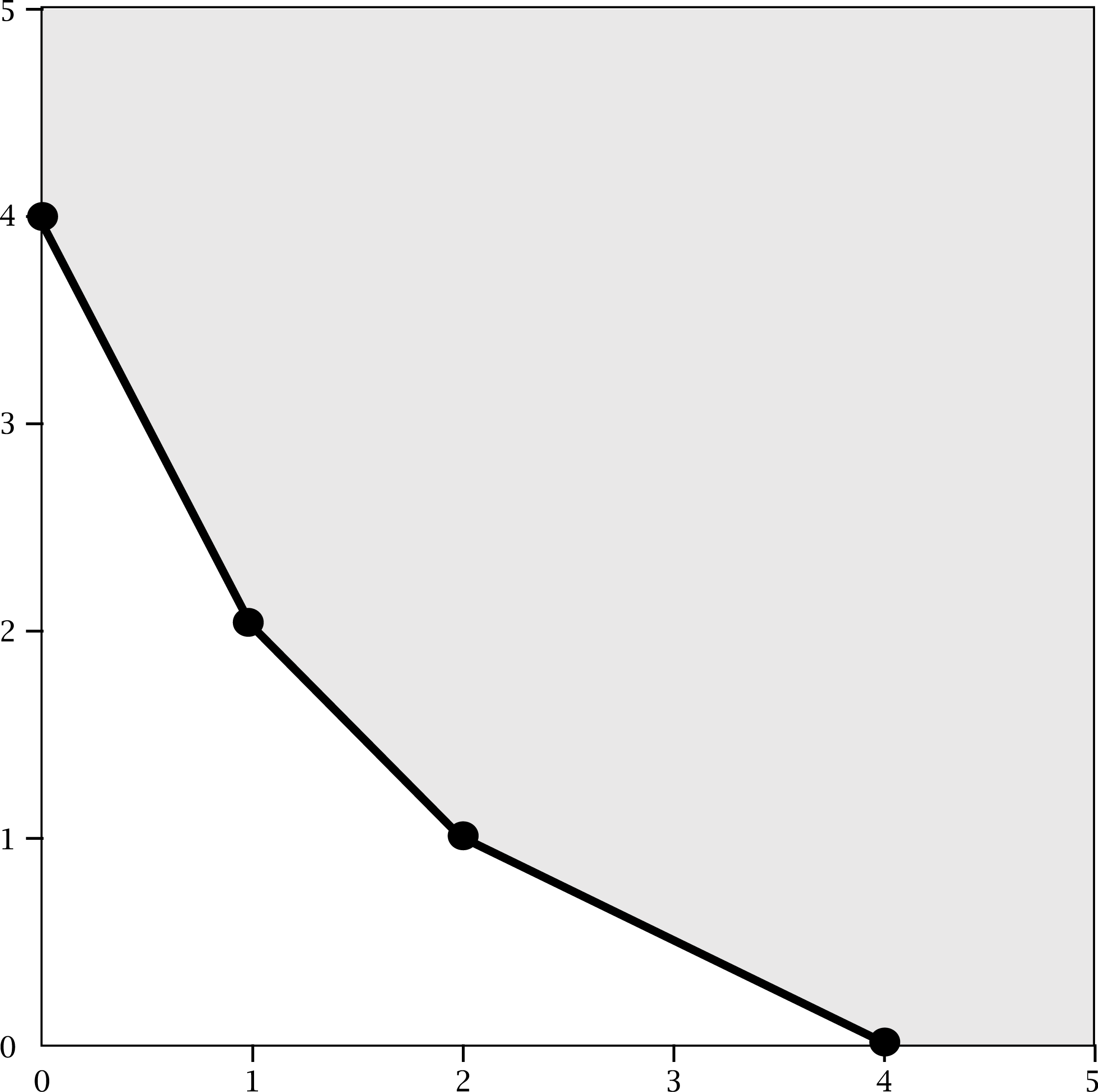}
\caption{Pareto-optimal set of Example \ref{example}.\label{fig:exsol}}
\end{center}
\end{figure}
\end{ex}

\section{Conclusions}

Due to the similarities between standard linear programming and multiobjective linear programming several authors wondered whether it would exist an interior point method valid to find the entire set of Pareto-optimal solutions of MOLP. This paper answers positively this question presenting an interior point based methodology to find that set. Our approach is constructive and we give an explicit sequence of $m$ SDP problems the solutions of which are the entire set of Pareto-optimal extreme points of MOLP. Moreover, we show how all these points can be obtained by applying the so called moment matrix algorithm (see \cite{lasserrebook,Laurent2007}). However, although our construction is explicit, we do not claim that it is computationally competitive with other currently available methods. The main drawback is the size of the SDP problems to be considered which is not polynomial in the input size of MOLP. The interest of our results is mainly theoretical because they show, as expected, the close relationship between scalar and multiobjective linear programming with regards to the solutions techniques. In addition, our results also show the powerfulness of some techniques developed in the field of polynomial optimization (\cite{lasserrebook,lasserre22}) to be applied in apparently different areas as multiobjetive optimization.

\end{document}